\documentclass[10pt]{amsart}
\usepackage{cmap}
\usepackage[cp1251]{inputenc}
\usepackage[english]{babel} % [russian,english]{babel} to make the title page with Russial journal title
\usepackage{amsmath}
\usepackage{amssymb}
\usepackage{amsfonts}
\usepackage{mathrsfs}

\def\udcs{515.1} %Here the author places classificators according to https://en.wikipedia.org/wiki/Universal_Decimal_Classification
\def\mscs{54A25} %Here the author places classificators according to the AMS classification list
\newtheorem{lemma}{Lemma}
\newtheorem{theorem}{Theorem}
\newtheorem{definition}{Definition}
\newtheorem{corollary}{Corollary}

\def\logo{{\bf\huge S\raisebox{0.2ex}{\hspace{0.55ex}\raisebox{0.05ex}e\hspace{-1.65ex}$\bigcirc$}MR}}

\def\semrtop
     {
  \vbox{\fussy
     \noindent\logo\hfill\raisebox{1ex}{ISSN 1813-3304}\par
     \vspace{5mm}
     \begin{center}
     {\huge SIBIRSKIE \ \`ELEKTRONNYE} \\[2mm]  %use this if no cyrillic fonts in your system
     {\huge MATEMATICHESKIE IZVESTIYA} \\[2mm]
     {\large Siberian Electronic Mathematical Reports} \\[1mm]
     {\LARGE\tt{http://semr.math.nsc.ru}}\\[0.5mm]
     \end{center}
     \vspace{-3mm}
     \noindent
     \begin{tabular}{c}
     \hphantom{aaaaaaaaaaaaaaaaaaaaaaaaaaaaaaaaaaaaaaaaaaaaaaaaaaaaaaaaaaaaaaaaaaaaaa} \\
     \hline\hline
     \end{tabular}
     \par\vspace{1mm}
     {\flushleft\it Vol. 14, % Vol., pp. in English, (to be replaced to Russian \CYRT\cyro\cyrm\ , \cyrs\cyrt\cyrr.  )
     pp. 1444--1444 (2017)   % filled by Editorial board
     \hfill{\rm\small UDC \udcs}} \newline % UDC in English, (to be replaced to Russian \CYRT\cyro\cyrm\ , \cyrs\cyrt\cyrr.  )
         {\rm\small DOI~10.17377/semi.2017.14.xxx}\hfill{\rm\small MSC\ \ \mscs }
     \newline
     \noindent
     \begin{center}
     Special issue: Groups and Graphs, Metrics and Manifolds\,---\,G2M2 2017
     \end{center}
  }%\vbox
}

\begin{document}
\thispagestyle{empty}

\title[On resolvability of Lindel\"of generated spaces]{On resolvability of Lindel\"of generated spaces}
\author{{Maria~A.~Filatova}}%
\address{Maria Alexandrovna Filatova
\newline\hphantom{iii} Ural Federal University,
\newline\hphantom{iii} 19 Mira str.,
\newline\hphantom{iii} 620002, Yekaterinburg, Russia
\newline\hphantom{iii} Krasovskii Institute of Mathematics and Mechanics,
\newline\hphantom{iii} 16 S.Kovalevskaya str.
\newline\hphantom{iii} 620990, Yekaterinburg, Russia}%
\email{MA.Filatova@urfu.ru}%

\author{{Alexander~V.~Osipov}}%
\address{Alexander Vladimirovich Osipov
\newline\hphantom{iii} Krasovskii Institute of Mathematics and Mechanics,
\newline\hphantom{iii} 16 S.Kovalevskoy str.,
\newline\hphantom{iii} 620990, Yekaterinburg, Russia;
\newline\hphantom{iii} Ural Federal  University,
\newline\hphantom{iii} 19 Mira str.,
\newline\hphantom{iii} 620002, Yekaterinburg, Russia;
\newline\hphantom{iii} Ural State University of Economics,
\newline\hphantom{iii} 62, 8th of March str.,
\newline\hphantom{iii} 620219, Yekaterinburg, Russia.}%
\email{oab@list.ru}%

\thanks{\sc Filatova, M.A., and Osipov, A.V.,
On resolvability of Lindel\"of generated spaces}
\thanks{\copyright \ 2017 Filatova, M.A., and Osipov, A.V.}
\thanks{\rm The work is supported by the Russian Academic Excellence Project (agreement no.
02.A03.21.0006 of August 27, 2013, between the Ministry of Education and Science of the Russian Federation and Ural Federal University)
}
\thanks{\it Received  September 1, 2017, published  December 31,  2017.}%

\semrtop \vspace{1cm}
\maketitle {\small
\begin{quote}
\noindent{\sc Abstract.} In this paper we study the properties of
$\mathscr{P}$ generated spaces (by analogy with compactly
generated). We prove that a regular Lindel$\ddot{o}$f generated
space with uncountable dispersion character is resolvable. It is
proved that Hausdorff hereditarily $L$-spaces are $L$-tight spaces
which were defined by Istv\'{a}n Juh\'{a}sz, Jan van Mill in ({\it
Variations on countable tightness},  arXiv:1702.03714v1). We also
prove $\omega$-resolvability of regular $L$-tight space with
uncountable dispersion character.

 \end{quote}
}

\section{Introduction}

A Hausdorff space is said to be a $k$-space (also called compactly
generated) if it has the final topology with respect to all
inclusions $K \mapsto X$ of compact subspaces $K$ of $X$, so that
a set $A$ in $X$ is closed in $X$ if and only if $A \cap K$ is
closed in $K$ for all compact subspaces $K$ of $X$. A space is
Fr$\acute{e}$chet-Uryson if, whenever a point $x$ is in the
closure of a subset $A$, there is a sequence from $A$ converging
to $x$; it is proved in \cite{A_1} that a space is hereditarily
$k$, i.e., every subspace is a $k$-space, if and only if it is
Fr$\acute{e}$chet-Uryson. For examples, any locally compact, or
the first countably Hausdorff spaces are $k$-spaces.

An interesting common generalizations of $k$-spaces and the notion
of tightness was given by A.V.~Arhangel'skii and D.N.~Stavrova in
\cite{A_S}. They defined the $k$, $k_1$, $k^{\ast}$-tightness as
follows: the $k$-tightness of $X$ does not exceed $\tau$ $(t_k
(X)\leq\tau )$ if and only if for every $A\subseteq X$ that is not
closed there exists a $\tau-$compact $B\subseteq X$ for which
$A\cap B$ is not closed in $X$ (a set $B$ is called $\tau-$compact
if $B= \bigcup \{ B_{\alpha} :$  $\alpha \in \tau \},$ where
$B_{\alpha}$ is a compact subset of $X$ for all $\alpha \in
\tau);$ the $k_1-$tightness of $X$ does not exceed $\tau$
$(t_{k_1} (X)\leq\tau )$ if and only if for every $A\subseteq X$
and every $x\in \overline{A}$ there exists a $\tau-$compact
$B\subseteq X$ such that  $x\in \overline{A\cap B};$ the
$k^\ast-$tightness of $X$ does not exceed $\tau$ $(t^{\ast}_k
(X)\leq\tau )$ if and only if for every $A\subseteq X$ and every
$x\in \overline{A}$ there exists a $\tau-$compact $B\subseteq A$
such that is $x\in \overline{B}.$

In \cite{J_M} Istv\'{a}n Juh\'{a}sz and Jan Van Mill defined and
studied nine attractive natural tightness conditions for
topological spaces. They called a space $\mathscr{P}-$tight, if
for all $x\in X$ and $A\subseteq X$ such that $x\in \overline{A},$
there exists $B\subseteq A$ such that $x\in \overline{B}$ and $B$
has the property~$\mathscr{P}.$

Istv\'{a}n Juh\'{a}sz and Jan Van Mill considered in \cite{J_M} the following properties $\mathscr{P}$
that a subspace of a topological space might have :

\smallskip

\begin{tabular}{r l}
  $\omega$D & Countable discrete; \\
  $\omega$N & Countable and nowhere dense; \\
  $C_2$ & Second-countable; \\
  $\omega$ & Countable; \\
  hL & Hereditarily Lindel\"of; \\
  $\sigma$-cmp & $\sigma-$compact; \\
  ccc & The countable chain condition; \\
  L & Lindel\"of;\\
  wL & Weakly Lindel\"of. \\
\end{tabular}

\medskip

Inspired by the researches above, in the first part of this paper
we introduce and study $\mathscr{P}$ generated spaces (by analogy
with $k$-spaces). Since the space with property $\mathscr{P}$ is
not necessarily closed,  there are two ways to determine the
$\mathscr{P}$ generated spaces.
\begin{definition}\label{P-s}
A topological space $X$ is $\mathscr{P}-$space ($\mathscr{P}$ generated space) if a subspace $A$ is closed in $X$ if and only if
$A \cap P$ is closed in $P$ for any subspace $P \subseteq X$ which has the property $\mathscr{P}.$
\end{definition}
\begin{definition}\label{Pc-s}
A topological space $X$ is $\mathscr{P}c-$space if  a set $A$ is
closed in $X$ if and only if $A \cap P$ is closed in $P$ (or, is
the same one, in $X$) for any closed subspace $P\subseteq X$ which
has the property $\mathscr{P}.$
\end{definition}

In this paper we will consider the property $\mathscr{P}\in
\{\omega$N, $C_2$, $\omega$, $hL$, $\sigma$-$cmp$, $ccc$, $L$, $wL
\}$ because each $\omega$D-space is discrete. Note that every
space with countable tightness (i.e. $\omega$-tight) is
$\omega-$space \cite{Mic}.
\smallskip

 The class of $L$-spaces generalizes the class
of spaces with countable tightness and the class of $k$-spaces.
Below (Theorem \ref{Ex1}) we get the example of the space with
countable tightness, but is not $Lc$-space.

\smallskip
Note that the relationships between the spaces of
$\mathscr{P}-$tight where the property $\mathscr{P}\in \{\omega$N,
$\omega$D, $C_2$, $\omega$, $hL$, $\sigma$-$cmp$, $ccc$, $L$, $wL
\}$ were considered in \cite{J_M}.

We summarize the relationships between $\mathscr{P}-$spaces
($\mathscr{P}-$s) and  $\mathscr{P}-$tight ($\mathscr{P}-$t) where
the property $\mathscr{P}\in \{\omega$N, $\omega$, $hL$,
$\sigma$-$cmp$, $ccc$, $L$, $wL \}$ in next diagrams.

\begin{center}
% This is a LaTeX picture output by TeXCAD.
% File name: [Clipboard].
% Version of TeXCAD: 4.5
% Reference / build: 13-Sep-2015 (rev. a64)
% For new versions, check: http://texcad.sf.net/
% Options on the following lines.
%\grade{\on}
%\emlines{\off}
%\epic{\off}
%\beziermacro{\on}
%\reduce{\on}
%\snapping{\off}
%\pvinsert{% Your \input, \def, etc. here}
%\quality{8.000}
%\graddif and only if{0.005}
%\snapasp{1}
%\zoom{6.7272}
\unitlength 1mm % = 2.845pt
\linethickness{0.4pt}
\ifx\plotpoint\undefined\newsavebox{\plotpoint}\fi % GNUPLOT compatibility
\begin{picture}(86.069,66.596)(0,0)
\put(35.676,0){\makebox(0,0)[cc]{wL-s}}
\put(61.244,13.23){\makebox(0,0)[cc]{L-s}}
\put(86.069,28.244){\makebox(0,0)[cc]{$\sigma$-cmpt-s}}
\put(83.839,66.15){\makebox(0,0)[cc]{$\sigma$-cmpt-t}}
\put(60.055,52.028){\makebox(0,0)[cc]{L-t}}
\put(24.825,52.028){\makebox(0,0)[cc]{hL-t}}
\put(24.973,13.825){\makebox(0,0)[cc]{hL-s}}
\put(51.731,28.838){\makebox(0,0)[cc]{$\omega$-s}}
%\vector[middle](83.839,63.623)(83.839,31.217)
\put(83.839,47.42){\vector(0,-1){.07}}\put(83.839,63.623){\line(0,-1){32.4059}}
%\end
%\vector[middle](1.487,36.568)(1.487,4.46)
\put(1.487,20.514){\vector(0,-1){.07}}\put(1.487,36.568){\line(0,-1){32.1086}}
%\end
%\vector[middle]{dash}{1}(21.108,11.892)(4.014,2.676)
\put(12.561,7.284){\vector(-2,-1){.07}}\multiput(21.038,11.822)(-.0581458,-.0313481){14}{\line(-1,0){.0581458}}
\multiput(19.41,10.944)(-.0581458,-.0313481){14}{\line(-1,0){.0581458}}
\multiput(17.782,10.066)(-.0581458,-.0313481){14}{\line(-1,0){.0581458}}
\multiput(16.154,9.189)(-.0581458,-.0313481){14}{\line(-1,0){.0581458}}
\multiput(14.526,8.311)(-.0581458,-.0313481){14}{\line(-1,0){.0581458}}
\multiput(12.898,7.433)(-.0581458,-.0313481){14}{\line(-1,0){.0581458}}
\multiput(11.27,6.555)(-.0581458,-.0313481){14}{\line(-1,0){.0581458}}
\multiput(9.642,5.678)(-.0581458,-.0313481){14}{\line(-1,0){.0581458}}
\multiput(8.013,4.8)(-.0581458,-.0313481){14}{\line(-1,0){.0581458}}
\multiput(6.385,3.922)(-.0581458,-.0313481){14}{\line(-1,0){.0581458}}
\multiput(4.757,3.044)(-.0581458,-.0313481){14}{\line(-1,0){.0581458}}
%\end
%\vector[middle]{dash}{1}(24.825,48.312)(24.973,18.433)
\put(24.899,33.372){\vector(0,-1){.07}}\put(24.754,48.241){\line(0,-1){.996}}
\put(24.764,46.249){\line(0,-1){.996}}
\put(24.774,44.257){\line(0,-1){.996}}
\put(24.784,42.265){\line(0,-1){.996}}
\put(24.794,40.274){\line(0,-1){.996}}
\put(24.804,38.282){\line(0,-1){.996}}
\put(24.814,36.29){\line(0,-1){.996}}
\put(24.824,34.298){\line(0,-1){.996}}
\put(24.834,32.306){\line(0,-1){.996}}
\put(24.844,30.314){\line(0,-1){.996}}
\put(24.854,28.322){\line(0,-1){.996}}
\put(24.863,26.33){\line(0,-1){.996}}
\put(24.873,24.338){\line(0,-1){.996}}
\put(24.883,22.346){\line(0,-1){.996}}
\put(24.893,20.354){\line(0,-1){.996}}
%\end
%\vector[middle](56.785,66.298)(75.812,66.298)
\put(66.298,66.298){\vector(1,0){.07}}\put(56.785,66.298){\line(1,0){19.0273}}
%\end
%\vector[middle](28.987,52.325)(55.447,52.325)
\put(42.217,52.325){\vector(1,0){.07}}\put(28.987,52.325){\line(1,0){26.46}}
%\end
%\vector[middle](5.054,.743)(31.514,.743)
\put(18.284,.743){\vector(1,0){.07}}\put(5.054,.743){\line(1,0){26.46}}
%\end
%\vector[middle]{dash}{1}(29.879,13.973)(50.839,13.973)
\put(40.359,13.973){\vector(1,0){.07}}\put(29.809,13.903){\line(1,0){.9981}}
\put(31.805,13.903){\line(1,0){.9981}}
\put(33.801,13.903){\line(1,0){.9981}}
\put(35.797,13.903){\line(1,0){.9981}}
\put(37.793,13.903){\line(1,0){.9981}}
\put(39.789,13.903){\line(1,0){.9981}}
\put(41.786,13.903){\line(1,0){.9981}}
\put(43.782,13.903){\line(1,0){.9981}}
\put(45.778,13.903){\line(1,0){.9981}}
\put(47.774,13.903){\line(1,0){.9981}}
\put(49.77,13.903){\line(1,0){.9981}}
%\end
\put(0,.446){\makebox(0,0)[cc]{ccc-s}}
\put(34.784,38.055){\makebox(0,0)[cc]{wL-t}}
%\vector[middle](61.096,48.906)(61.096,17.987)
\put(61.096,33.446){\vector(0,-1){.07}}\put(61.096,48.906){\line(0,-1){30.919}}
%\end
\put(1.784,39.244){\makebox(0,0)[cc]{ccc-t}}
%\vector[middle](5.649,39.541)(32.109,39.541)
\put(18.879,39.541){\vector(1,0){.07}}\put(5.649,39.541){\line(1,0){26.46}}
%\end
%\vector[middle]{dash}{1}(47.271,27.203)(30.027,17.838)
\put(38.649,22.521){\vector(-2,-1){.07}}\multiput(47.201,27.133)(-.0586514,-.0318538){14}{\line(-1,0){.0586514}}
\multiput(45.558,26.241)(-.0586514,-.0318538){14}{\line(-1,0){.0586514}}
\multiput(43.916,25.349)(-.0586514,-.0318538){14}{\line(-1,0){.0586514}}
\multiput(42.274,24.457)(-.0586514,-.0318538){14}{\line(-1,0){.0586514}}
\multiput(40.632,23.565)(-.0586514,-.0318538){14}{\line(-1,0){.0586514}}
\multiput(38.989,22.673)(-.0586514,-.0318538){14}{\line(-1,0){.0586514}}
\multiput(37.347,21.781)(-.0586514,-.0318538){14}{\line(-1,0){.0586514}}
\multiput(35.705,20.889)(-.0586514,-.0318538){14}{\line(-1,0){.0586514}}
\multiput(34.063,19.998)(-.0586514,-.0318538){14}{\line(-1,0){.0586514}}
\multiput(32.421,19.106)(-.0586514,-.0318538){14}{\line(-1,0){.0586514}}
\multiput(30.778,18.214)(-.0586514,-.0318538){14}{\line(-1,0){.0586514}}
%\end
%\vector[both]{dash}{1}(51.433,61.69)(51.582,32.555)
\put(51.582,32.555){\vector(0,-1){.07}}\put(51.433,61.69){\vector(0,1){.07}}\put(51.363,61.62){\line(0,-1){.9712}}
\put(51.373,59.677){\line(0,-1){.9712}}
\put(51.383,57.735){\line(0,-1){.9712}}
\put(51.393,55.793){\line(0,-1){.9712}}
\put(51.403,53.85){\line(0,-1){.9712}}
\put(51.412,51.908){\line(0,-1){.9712}}
\put(51.422,49.966){\line(0,-1){.9712}}
\put(51.432,48.023){\line(0,-1){.9712}}
\put(51.442,46.081){\line(0,-1){.9712}}
\put(51.452,44.138){\line(0,-1){.9712}}
\put(51.462,42.196){\line(0,-1){.9712}}
\put(51.472,40.254){\line(0,-1){.9712}}
\put(51.482,38.311){\line(0,-1){.9712}}
\put(51.492,36.369){\line(0,-1){.9712}}
\put(51.502,34.427){\line(0,-1){.9712}}
%\end
%\vector[middle](57.082,10.554)(40.73,2.23)
\put(48.906,6.392){\vector(-2,-1){.07}}\multiput(57.082,10.554)(-.0662008009,-.0337022259){247}{\line(-1,0){.0662008009}}
%\end
%\vector[middle](79.826,23.784)(63.474,15.46)
\put(71.65,19.622){\vector(-2,-1){.07}}\multiput(79.826,23.784)(-.0662024291,-.0337004049){247}{\line(-1,0){.0662024291}}
%\end
%\vector[middle](54.258,48.311)(37.906,39.987)
\put(46.082,44.149){\vector(-2,-1){.07}}\multiput(54.258,48.311)(-.0662024291,-.0337004049){247}{\line(-1,0){.0662024291}}
%\end
%\vector[middle](19.92,49.946)(3.568,41.622)
\put(11.744,45.784){\vector(-2,-1){.07}}\multiput(19.92,49.946)(-.0662024291,-.0337004049){247}{\line(-1,0){.0662024291}}
%\end
%\vector[middle](44.298,63.325)(27.946,55.001)
\put(36.122,59.163){\vector(-2,-1){.07}}\multiput(44.298,63.325)(-.0662024291,-.0337004049){247}{\line(-1,0){.0662024291}}
%\end
%\vector[middle](80.123,61.69)(63.771,53.366)
\put(71.947,57.528){\vector(-2,-1){.07}}\multiput(80.123,61.69)(-.0662024291,-.0337004049){247}{\line(-1,0){.0662024291}}
%\end
%\vector[middle]{dash}{1}(56.487,29.284)(77.446,29.284)
\put(66.967,29.284){\vector(1,0){.07}}\put(56.417,29.214){\line(1,0){.998}}
\put(58.413,29.214){\line(1,0){.998}}
\put(60.409,29.214){\line(1,0){.998}}
\put(62.405,29.214){\line(1,0){.998}}
\put(64.401,29.214){\line(1,0){.998}}
\put(66.398,29.214){\line(1,0){.998}}
\put(68.394,29.214){\line(1,0){.998}}
\put(70.39,29.214){\line(1,0){.998}}
\put(72.386,29.214){\line(1,0){.998}}
\put(74.382,29.214){\line(1,0){.998}}
\put(76.378,29.214){\line(1,0){.998}}
%\end
%\vector[middle](36.271,33.892)(36.271,2.973)
\put(36.271,18.433){\vector(0,-1){.07}}\put(36.271,33.892){\line(0,-1){30.919}}
%\end
\put(48.46,66.596){\makebox(0,0)[cc]{$\omega$-t}}
\end{picture}

\medskip

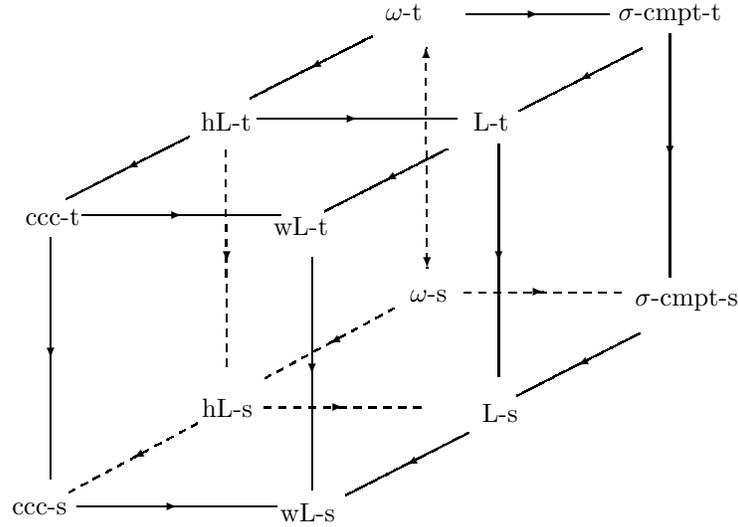
Fig.~1. The Diagram of the relationships between \\
$\mathscr{P}-$spaces and $\mathscr{P}-$tight.

\end{center}

In 1943 E.~Hewitt \cite{Hu43} called a topological space
$\tau$-resolvable, if it can be represented as a union of $\tau$
dense disjoint subsets. A 2-resolvable space is called resolvable,
and irresolvable space is one which is not resolvable. He also
defined the dispersion character $\Delta(X)$ of a space $X$  as
the smallest size of a non-empty open subset of $X.$ A topological
space $X$ is called maximally resolvable if it is
$\Delta(X)-$resolvable.

The $\omega$-resolvability of Lindel\"of spaces whose dispersion character is uncountable was proved by I.~Juhasz, L.~Soukup, Z.~Szentmiklossy  in \cite{JSS}.

E.~Hewitt in \cite{Hu43} constructed an example of countable
irresolvale normal space, so condition $\Delta(X)>\omega$ is
natural. V.I.~Malykhin constructed an example of irresolvale
Hausdorf Lindel\"of space with uncountable dispersion character in
\cite{Mal98}, therefore, resolvability of regular Lindel\"of
spaces was studied.

The resolvability of locally compact spaces was proved by Hewitt
in 1943 \cite{Hu43}. The resolvability (maximal resolvability) of
$k$-spaces was proved by N.V.~Velichko in 1976 \cite{Vel}
(E.G.~Pytkeev in 1983 \cite{Pyt}). In connection with these
results, the question of the resolvability of $L$-spaces is
natural.

Second part of this paper is devoted to resolvability of regular
$L$-spaces of uncountable dispersion character. We also prove
$\omega$-resolvability of regular hereditarily $L$-spaces of
uncountable dispersion character.

Throughout this paper the symbol $\omega$ denotes the smallest infinite cardinal,
$\omega_1$ stands for the smallest uncountable cardinal. For a subset $A$ of a topological space $X$, the closure and
interior set of $A$ are respectively denotes by $\overline{A}$ (or $[A]$) and $Int(A).$
We assume that all spaces are Hausdorff. Notation and terminology are taken from \cite{Eng}.

%\begin{theorem}
%Text of the theorem
%\end{theorem}

%\begin{proof}
%Text of the proof.
%\end{proof}

\section{ $\mathscr{P}$ generated space}

By Definitions \ref{P-s} and \ref{Pc-s},  a $\mathscr{P}c-$space
is a $\mathscr{P}-$space. The following example shows that the
converse is not true.

\begin{theorem}\label{Ex1}

There exists a space $X$ such that $X$ is a $\omega-$space and
$t_k (X)=t_{k_1} (X)= t^{\ast}_k (X))=t(X)= \omega$.
 \end{theorem}

\begin{proof}
  Let $X = \beta \mathbb{N},$ $M = \beta \mathbb{N} \setminus \mathbb{N}.$
In a $x\in X$ we put the base of neighborhoods ${\mathscr B}(x) =
\{ x \},$ if $x\in \mathbb{N}$ and ${\mathscr B}(x) = \{ (U(x)
\setminus M) \cup \{ x \} \}$ where $U(x)$ is open in $\beta
\mathbb{N},$ if $x\in M.$

Now we show that $X$ is a $\omega-$space. Consider the set
$A\subset X$ for which $A \cap P$ is closed in $P$ for any
countable subspace $P \subseteq X.$ Let $x\in \overline{A}.$ The
set $P_1 = (A \cap \mathbb{N}) \cup \{ x \}$ is a countable. Then
$A \cap P_1$ is closed in $P_1$ therefore $x\in A.$

Let us note that if $B$ is infinity subset of $\mathbb{N}$ then $\overline{B}$ contains  uncountable discrete space. Consequently $\overline{B}$ is not a countable space.
Therefore, if $B\subset X$ is closed countable space, then the intersection
$\mathbb{N}\cap B$ is finite, and, consequently, closed in $X.$ But $\mathbb{N}$ is not closed in $X,$ hence, $X$ is not $\omega c-$space.

Being the $\omega-$space, $X$ is $hL$-space, $\sigma$-$cmp$-space,
$ccc$-space, L-space and $wL$-space. It is clear that $X$ is not
$hL$$c-$space, $\sigma-$cmp$c-$space, $ccc$$c-$space,
$L$$c-$space, $wL$$c-$space.

The equalities $t_k (X)=t_{k_1} (X)= t^{\ast}_k (X))=t(X)= \omega$
are obvious. Being $\omega-$tight a space $X$ is $wL$-tight too.
\end{proof}

\begin{corollary} There exists a space $X$ such that $X$ is a $\omega-$space, but is not $\omega c-$space.
\end{corollary}

It is clear that any $k-$space is $Lc-$space (hence, $L$-space),
and every space with countable tightness is $hL$-space (therefore
$\sigma-$cmp-space, $ccc$-space, $L$-space and $wL$-space). The
converse statements are not true. Theorem \ref{Ex1} shows that
even $wL$$c-$spaces do not generalize a notion of countable
tightness.
 Thus, a natural generalization of spaces with countable tightness and $k$-spaces is $\mathscr{P}-$spaces.

\begin{theorem}\label{Ex2} There exists a space $X$ such that $X$
is a $L$-space, but is not $hL$-tight.
\end{theorem}

\begin{proof}

For example, let $D$ be the infinite discrete space of cardinality
continuum $\mathfrak{c}$, and $X=\beta D$ be the
$\check{C}$ech-Stone compactification of $D$. Being the compact
space, $X$ is a $L$-space. Consider $p\in X\setminus \{\bigcup
\overline{A} : A\in [D]^{\leq \omega} \}$. Since the space $D$ has
not an uncountable Lindel\"of subspaces, there is no Lindel\"of
subspace $B\subset D$ such that $p\in \overline{B}$.

\end{proof}

In what follows we shall concentrate on the study of some
properties of $\mathscr{P}-$spaces. Most of all we are interested
in the resolvability of $L$-space. The following statements will
be useful for these purpose.
\begin{theorem}\label{F} Let the property $\mathscr{P}\in \{\omega$N, $\omega$, $hL$, $\sigma$-$cmp$, $ccc$, $L$, $wL
\}$. Then the following properties of a space $X$ are equivalent.

  (i)  $X$ is a $\mathscr{P}-$space;

  (ii) a set $A\subset X$ is non-closed in $X$ if and only if
  there exists $P\subseteq X$ which  have a property $\mathscr{P}$ such that $P\cap A$ is non-closed in $X.$
\end{theorem}
\begin{proof}
  $(i)\rightarrow (ii).$ Let $X$ is a $\mathscr{P}-$space, $A\subset X$ is non-closed in $X.$ Then there exists $P\subseteq X$ which  have a property $\mathscr{P}$ such that $P\cap A$ is non-closed in $P,$ therefore $P\cap A$  is non-closed in $X.$
  Let for a set $A$ there exists $P\subseteq X$ which  have
  a property~$\mathscr{P}$ such that $P\cap A$ is non-closed in $X.$
  Consider
 $x \in \overline{P\cap A} \setminus (P\cap A) \subseteq \overline{P}\cap \overline{A} \setminus (P\cap A).$
 If $x\in A$ then $x \in \overline{P}\setminus P,$, i.e. $P\cap A$ is non-closed in $P,$ consequently a set $A$ is non-closed in $X.$
 If $x\in P$ then $x \in \overline{A}\setminus A$, i.e. a set $A$ is non-closed in $X.$

 $(ii)\rightarrow (i).$ Consider $A\subset X$ such that $A \cap P$ is closed in $P$ for any subspace
 $P \subseteq X$ which has the property $\mathscr{P}.$ Suppose
 that $A\neq \overline{A}$. Then exists $P\subseteq X$ which  have a property $\mathscr{P}$ such that $P\cap A$ is non-closed in
 $X.$ Consider $x\in \overline{P\cap A}\setminus (P\cap A)$. Let
 $P_1=P\cup \{x\}$. Then $P_1$ has the property $\mathscr{P}$ and
 $P_1\cap A$ is not closed in $P_1$, contradiction.
\end{proof}

\begin{corollary}\label{col_1}
  A space X is $\mathscr{P}-$space if and only if for any $A$ non-closed in $X$
  there are $x \in \overline{A} \setminus A$ and $P\subseteq X$
  with a property~$\mathscr{P}$
  such that $x\in \overline{P\cap A}.$
\end{corollary}

\begin{corollary}\label{col_2}
   For any non-isolated point $x$ in $\mathscr{P}-$space $X$
   there is a subspace $P\subseteq X$ with a property $\mathscr{P}$
   such that $x\in \overline{P\setminus \{ x \}}.$
\end{corollary}

We give the following definition.
\begin{definition}\label{ISC}
 The property $\mathscr{P}$ is ICS  (independent of the containing subspace) in $X$ if
 $P\subseteq X$ has a property $\mathscr{P}$ in $X$ if and only if
 $P$ has a property $\mathscr{P}$ in $Y$ for any
 subset $Y$ such that $P\subseteq Y.$
\end{definition}

%Note that from all nine properties listed in the Introduction,
%only property $\omega$N is not ICS.

      The next theorem shows that hereditarily $\mathscr{P}-$space is $\mathscr{P}-$tight if property $\mathscr{P}$ is ICS in $X.$
\begin{theorem}
Let property $\mathscr{P}$ be ICS in $X.$ A space $X$ is
hereditarily $\mathscr{P}-$space if and only if
  for any $A\subset X$ and $x \in \overline{A} \setminus A$ there exists
  $P\subseteq A$ which have a property $\mathscr{P}$
  such that $x\in \overline{P}.$
\end{theorem}
\begin{proof}
  Let $X$ be a hereditarily $\mathscr{P}-$space, $A\subset X$ and $x \in \overline{A} \setminus A.$
  The subspace $B=A\cup \{ x\}$ is $\mathscr{P}-$space, $x$ is non-isolated point in $B.$
  Then there is $P\subseteq A$ with a property $\mathscr{P}$ such that $x\in \overline{P}.$

  Converse, let $B\subset X,$ $A\subseteq B$ and $A\cap P$ is closed in $B$
  for any $P\subseteq B$ which have a property $\mathscr{P}.$
  If $x\in \overline{A}\setminus A$ then
  there exists $P\subseteq A$ which have a property $\mathscr{P}$
  such that $x\in \overline{P},$ i.e. $A\cap P$ is non-closed in $B.$
\end{proof}
\begin{corollary}
Let property $\mathscr{P}$ be ICS in $X.$  A space $X$ is
$\mathscr{P}-$tight if and only if each subspace $Y$ of $X$ is
$\mathscr{P}-$tight.
\end{corollary}

Note that $\omega$N-tight is not hereditarily property. The real
numbers $\mathbb{R}$ is $\omega$N-tight, but it subspace
$Y=\{\frac{1}{n}: \ n\in \mathbb{N} \}\cup\{0\}$ is not
$\omega$N-tight.

\section{On resolvability of $L$-spaces}

The maximal resolvability of $\omega-$tight spaces whose
dispersion character is uncountable was proved by E.G.~Pytkeev in
\cite{Pyt}. Therefore $\omega-$spaces with uncountable dispersion
character  are maximally resolvable. A.~Bella and V.I.~Malykhin
was proved that $\omega$N-tight spaces are maximally resolvable
\cite{Bella_Mal}.

E.~Hewitt in \cite{Hu43} constructed an example of countable
irresolvable normal space. Such space is obviously $\omega-$tight
($hL$-, $\sigma$-$cmp$-, $ccc$-, $L$-, $wL$-tight) space with
countable dispersion character. Therefore, in studying the
resolvability of these spaces, it is necessarily to consider the
spaces with uncountable dispersion character.

In this section we prove $\omega-$resolvability of regular
$L$-tight spaces which dispersion character is uncountable.
Consequently, $hL$-, $\sigma$-$cmp$-, $ccc$-, $wL$-tight spaces
are also $\omega-$resolvable (if $X$ is regular space then
properties $ccc$-tight, $L$-tight, $wL$-tight are equivalent).

The $\omega-$resolvability of $\omega$D-spaces was proved by
P.L.~Sharma and S.~Sharma in~\cite{Sharma}. V.I.~Malychin
constructed an example of irresolvable Hausdorf Lindel$\ddot{o}$f
space with uncountable dispersion character \cite{Mal98}.

It is clear that any Lindel$\ddot{o}$f space is $L$-space.
Therefore, the resolvability of $L$-spaces must be investigated in
the class of regular spaces.

\medskip

We will use a following

\begin{theorem}\label{th20} (Hewitt's criterion of resolvability \cite{Bella_Mal})
A topological space $X$ is resolvable ($\tau$-resolvable ) if and
only if for all open subset $U$ of $X$ there exist nonempty
resolvable ($\tau$-resolvable) subspace without isolated points.
\end{theorem}

\medskip
  A.G. El'kin  proved the following elegant result

\begin{theorem}(Theorem 1 in  \cite{Elk})\label{th11} Let $X$ be a
collectionwise Hausdorff $\sigma$-discrete normal space that
satisfies the following condition:

(*) For each point $x\in X$ there exists a discrete set $D\subset
X$ such that $x\in \overline{D}\setminus D$. Then the space $X$ is
$\omega$-resolvable.

\end{theorem}

Recall that a set $D\subseteq X$ is strongly discrete if for every
$x\in D$ there is an open neighborhood $U_x$ such that $U_x\cap
U_y=\emptyset$ for $x\neq y$.

\medskip

A point $x$ of a space $X$ is called $lsd$-point ({\it a limit
point of a strongly discrete subspace}) if there exists a strongly
  discrete subspace $D\subset X$ such that $x \in \overline{D}\setminus D$ \cite{Mal98}.

\medskip

By using idea of the proof of El'kin's Theorem \ref{th11},
P.L.~Sharma and S.~Sharma proved
 the following result \cite{Sharma}.

\begin{theorem}\label{Sharma} (P.L.~Sharma, S.~Sharma) Let $X$ be a $T_1$-space.
  If each $x\in X$ is a limit point of a strongly discrete subspace of $X$ then $X$ is a $\omega-$resolvable.
\end{theorem}

\begin{definition} A discrete set $D$ of cardinality $\omega_1$ is a correct discrete, if any $Y\subseteq D$ such that $|Y|=\omega_1$ has a
$\omega_1$- accumulation point.
\end{definition}

The following lemma is a generalized version of the lemma 2.1 in \cite{Fi03}.

 \begin{lemma}\label{lem1} Let $X$ be a regular space such that for each $x\in X$ there exists a correct discrete set $D_x$ of cardinality $\omega_1$ such that $x$ is  a
 $\omega_1$-accumulation point of $D_x$. Then $X$ is resolvable.
 \end{lemma}

 \begin{proof}

 By Theorem \ref{th20}, it suffices to prove that each non-empty open set $V$ of $X$
contains dense-in-itself $Y$ such that
$\overline{Y}$ is resolvable.

 Denote by $P(D)$ the set of $\omega_1$-accumulation
   points of a correct discrete set $D$. Let us note that $\overline{P(D)}=P(D)$ and for any open set
   $U$ such that $P(D)\subset U$, the set $D\setminus U$ is at most
   countable.

   Let $y \in V$, $D_y \subset V$ is a correct discrete set of
cardinality $\omega_1$, such that $y \in P(D_y)$.  Let $Y_1 =
D_y$, $Y_n = \bigcup \{ D_y: y \in Y_{n-1}$,  where $D_y$ is
correct discrete set, $y$ is $\omega_1$-accumulation of $ D_y \}$, we put
$Y = \bigcup\limits_{n=1}^{\infty} Y_n$. It is clear that $|Y|=\omega_1$ and
any $y\in Y$ is a $\omega_1$-accumulation point of correct (in
$\overline{Y}$) discrete subset $D_y$ of $Y$.

Now we prove that $\overline{Y}$ is resolvable.

 First, we renumber the space $Y=\{ y_{\alpha} : \alpha < \omega_1 \}$.

For $\alpha < \omega_1$ we construct sets $A_{\alpha}^i\subset \overline{Y}$, $A_i =
\bigcup \limits_{\alpha < \omega_1} A_{\alpha}^i $ where $i = 1,
2$, such that $A_1 \cap A_2 = \emptyset$ and $\overline{A_1} =
\overline{A_2} \supset Y$.

Consider $y_1$. Let $D_1$ be a correct (in $X$) discrete space of
cardinality $\omega_1$, such that $y_1 \in P(D_1)$. Let $A_1^1 =
D_1$, $A_1^2 = P(D_1)$, $P_1 = P(D_1)$. Obviously, $A_1^1\cap
A_1^2=\emptyset$ and $\overline{A_1^1} \cap \overline{A_1^2}
\supset P_1$.

Suppose that for each $\alpha < \beta < \omega_1$ we construct
sets $P_{\alpha}\subset \overline{Y}$, $D_{\alpha}\subset Y$, $A_{\alpha}^1$, $A_{\alpha}^2$
with the following properties.

{\bf 1. } If $y_{\alpha} \in \left[ \bigcup\limits_{\eta < \alpha}
P_{\eta}\right]$, then $P_{\alpha}=D_{\alpha}= \varnothing$,
$A_{\alpha}^i = \bigcup\limits_{\eta < \alpha} A_{\eta}^i$, $i =
1, 2$.

{\bf 2. } If $y_{\alpha} \notin \left[ \bigcup\limits_{\eta <
\alpha} P_{\eta}\right]$, but $y_{\alpha} \in \left[
\bigcup\limits_{\eta < \alpha} A_{\eta}^1 \right] \bigcap \left[
\bigcup\limits_{\eta < \alpha} A_{\eta}^2 \right]$, then $P_{\alpha} = \{ y_{\alpha} \}$, $D_{\alpha}=
\varnothing$, $A_{\alpha}^1 = \bigcup\limits_{\eta < \alpha}
A_{\eta}^1$,  $A_{\alpha}^2 =
\bigcup\limits_{\eta < \alpha} A_{\eta}^2$.

{\bf 3. } If $y_{\alpha} \notin \left[ \bigcup\limits_{\eta <
\alpha} P_{\eta}\right]$ and $y_{\alpha} \notin \left[
\bigcup\limits_{\eta < \alpha} A_{\eta}^1 \right] \bigcap \left[
\bigcup\limits_{\eta < \alpha} A_{\eta}^2 \right]$, then $D_{\alpha}$ is
a correct discrete set of cardinality $\omega_1$,
such that $y_\alpha \in P(D_{\alpha})$, $(\bigcup\limits_{\eta<\alpha} D_{\eta})\bigcap D_{\alpha}=\emptyset$, $P_{\alpha} =
P(D_{\alpha})$. In this case $A_{\alpha}^1 = \bigcup\limits_{\eta < \alpha} A_{\eta}^1\bigcup
P_{\alpha}$,  $A_{\alpha}^2 = \bigcup\limits_{\eta < \alpha}
A_{\eta}^2\bigcup D_{\alpha}$.

{\bf 4. } $A_{\alpha}^1 \cap A_{\alpha}^2 = \varnothing$;
$A_{\alpha}^i$ form a monotonically nondecreasing sequence (by
$\alpha$), $\left[ \bigcup\limits_{\eta \le \alpha}
P_{\eta}\right] \subset \left[ A_{\alpha}^i \right] $ , and
$A_{\alpha}^i \subset \bigcup\limits_{\eta \le \alpha} (P_{\eta}
\cup D_{\eta})$, $i = 1, 2$.

Consider $y_{\beta}$. If $y_{\beta}\in \left[
\bigcup\limits_{\alpha < \beta} P_{\alpha}\right]$, we suppose
$P_{\beta}=D_{\beta}= \varnothing$, $A_{\beta}^i =
\bigcup\limits_{\alpha < \beta} A_{\alpha}^i$, $i = 1, 2$. It is
easy to see $\left[ \bigcup\limits_{\alpha \le \beta}
P_{\alpha}\right] \subset \left[ A_{\beta}^i \right] $, $i = 1, 2$
and $A_{\beta}^1 \cap A_{\beta}^2 = \varnothing$.

Indeed, by construction, $ \bigcup\limits_{\alpha \le \beta}
\left[ P_{\alpha}\right] \subset \left[ A_{\beta}^i \right] $, $i
= 1, 2$, then $\left[  \bigcup\limits_{\alpha \le \beta}
P_{\alpha}\right] \subset \left[  \bigcup\limits_{\alpha \le
\beta} \left[ P_{\alpha}\right] \right] \subset \left[ A_{\beta}^i
\right] $.

Let $y_{\beta}\notin \left[ \bigcup\limits_{\alpha < \beta}
P_{\alpha}\right]$, but $y_{\beta} \in \left[
\bigcup\limits_{\alpha < \beta} A_{\alpha}^1 \right] \bigcap
\left[ \bigcup\limits_{\alpha < \beta} A_{\alpha}^2 \right]$, then
 $P_{\beta} = \{ y_{\beta} \}$, $D_{\beta} = \varnothing$,
$A_{\beta}^i = \bigcup\limits_{\alpha < \beta} A_{\alpha}^i$, $i =
1, 2$.

It is easy to see that $A_{\beta}^i$ satisfy item {\bf 4}.

Finally the last case $y_{\beta} \notin \left[
\bigcup\limits_{\alpha < \beta} P_{\alpha}\right]$, and $y_{\beta}
\notin \left[ \bigcup\limits_{\alpha < \beta} A_{\alpha}^1 \right]
\bigcap \left[ \bigcup\limits_{\alpha < \beta} A_{\alpha}^2
\right]$.

Suppose, for definiteness, $y_{\beta} \notin \left[
\bigcup\limits_{\alpha < \beta} A_{\alpha}^1 \right]$. Consider a
neighborhood $U_\beta$ of $y_\beta$, such that the closure of
$U_\beta$ is disjoint with $\left[ \bigcup\limits_{\alpha < \beta}
A_{\alpha}^1 \right]$ and $\left[ \bigcup\limits_{\alpha < \beta}
P_{\alpha} \right]$.

In this neighborhood there is at most a countable of points of the
set $\bigcup\limits_{\alpha < \beta} A_{\alpha}^2$ because it intersect with $U_\beta$ can only with $D_\alpha$, $\alpha < \beta$
(by {\bf 4}).

For each $\alpha$ such that $D_\alpha \ne \varnothing$ and
$D_\alpha$ is a  discrete, $P(D_\alpha)\subset \overline{Y} \setminus [U_{\beta}]$,
hence, $D_\alpha \bigcap U_\beta$ is at most a countable. It
follows that $\left( \bigcup\limits_{\alpha < \beta} D_\alpha
\right) \bigcap U_\beta$ has a cardinality $<\omega_1$ and, by
{\bf 4}, $\left( \bigcup\limits_{\alpha < \beta} A_{\alpha}^2
\right)\bigcap U_{\beta}$ is countable. Let $D_\beta$ be a correct discrete set such that $D_\beta\subset U_\beta$, $y_\beta \in
P(D_\beta)$ and $D_\beta\cap \left( \bigcup\limits_{\alpha <
\beta} A_{\alpha}^2 \right)=\emptyset$.

Let $P_\beta = P(D_\beta)$, $A_{\beta}^1 = D_{\beta} \bigcup
\left( \bigcup\limits_{\alpha < \beta} A_{\alpha}^1 \right)$,
$A_{\beta}^2 = P_{\beta} \bigcup \left( \bigcup\limits_{\alpha <
\beta} A_{\alpha}^2 \right)$.

By construction, $A_{\beta}^1$, $A_{\beta}^2$ are disjoint sets
and has the property {\bf 4}.

We prove that $Y \subset \left[ \bigcup\limits_{\alpha < \omega_1}
P_\alpha \right]$. Suppose that $Y \setminus \left[
\bigcup\limits_{\alpha < \omega_1} P_\alpha \right] \ne
\varnothing$, consider $y_\gamma \in Y \setminus \left[
\bigcup\limits_{\alpha < \omega_1} P_\alpha \right]$. Then
$P_\gamma = \varnothing$, because if $P_\gamma \ne \varnothing$,
then $y_\gamma \in P_\gamma$. By {\bf 1, 2, 3}, we have $y_\gamma
\in \left[ \bigcup\limits_{\alpha < \gamma} P_\alpha \right]
\subset \left[ \bigcup\limits_{\alpha < \omega_1} P_\alpha
\right]$, a contradiction.

Let $A_i = \bigcup\limits_{\alpha < \omega_1} A_{\alpha}^i $, $i =
1, 2$. By construction, $A_1 \cap A_2 = \varnothing$ and $[A_1] =
[A_2] \supset Y$. So we have that $\overline{Y}$ is resolvable.

 \end{proof}

\begin{corollary}\label{cor5} Let $X$ be a regular space, $Y=\{x\in X : \exists$ countable discrete set $D_x$ such that $x\in \overline{D_x}\setminus D_x$ or $\exists$ correct discrete $D_x$   of cardinality $\omega_1$  such that $x\in \overline{D_x}\setminus D_x \}$ and $\overline{Y}=X$. Then $X$ is resolvable.

\end{corollary}

\begin{proof}
  Let $A=\{x\in X:  \exists$ countable discrete set $D_x$ such that $x\in \overline{D_x}\setminus D_x \}$ and $B=Y\setminus A$.
  By Theorem \ref{Sharma}, the set $Int (A)$ is resolvable. Hence, there are disjoint sets $A_1$, $A_2\subset Int (A)$ such that $\overline{A_i}\supset Int (A)$ for $i=1,2$. If $Int (A)=\emptyset$ then we put $A_i=\emptyset$ for $i=1,2$.

  Now we prove that if $Int (B)\neq \emptyset$ then $Int (B)$ is resolvable. Let $W$ be an non-empty open set such that $\overline{W}\subset Int (B)$. Note that there is  a set $Y\subset W$ such that  $|Y|=\omega_1$ and
any $y\in Y$ is a $\omega_1$-accumulation point of correct (in
$\overline{Y}$ and also in $\overline{W}$) discrete subset $D_y$ of $\overline{Y}$. Then, by Lemma \ref{lem1}, $Int (B)$  is resolvable.

  Hence, there are disjoint sets $B_1$, $B_2\subset Int (B)$ such that $\overline{B_i}\supset Int (B)$ for $i=1,2$. If $Int (B)=\emptyset$ then we put $B_i=\emptyset$ for $i=1,2$.

 Let $Z=X\setminus(\overline{Int(A)}\cup \overline{Int(B)})$,
 $A_z=A\cap Z$, $B_z=B\cap Z$. Consider $C=Int(A_z\cup B_z)$. Let $C_1=A\cap C$, $C_2=B\cap C$ for
 $C\neq \emptyset$. Note that $A\cap B=\emptyset$, hence, $C_1\cap
 C_2=\emptyset$, $Int(C_1)=Int(C_2)=\emptyset$ and $C=C_1\cup
 C_2$. It follows that $C\subset \overline{C_i}$ for $i=1,2$.
 If $C=\emptyset$ then we assume $C_1=C_2=\emptyset$.

 Let $D=X\setminus(\overline{Int(A)}\cup \overline{Int(B)}\cup
 \overline{C})$. If $D=\emptyset$ we assume $D_1=(A\cup B)\cap D$
 and $D_2=(X\setminus(A\cup B))\cap D$. Note that $D_1$ and $D_2$
 are disjoint sets and are dense in $D$.
If $D=\emptyset$ then we assume $D_1=D_2=\emptyset$. Note that
$X=\overline{Int(A)}\cup \overline{Int(B)}\cup \overline{C}\cup
\overline{D}$. Finally, let $X_i=A_i\cup B_i\cup C_i\cup D_i$ for
$i=1,2$. It is clear $X=\overline{X_i}$ for $i=1,2$ and $X_1\cap
X_2=\emptyset$.

\end{proof}

The following theorem is the main result of this work.

\begin{theorem}\label{Main}
  The regular $L$-space with uncountable dispersion character is resolvable.
\end{theorem}

\begin{proof} Let $U$ be an open subset of $X$ such that each  Lindel$\ddot{o}$f subspace of $U$ at most countable. Note that if $U\neq \emptyset$ then $U$ is maximally resolvable  \cite{Pyt} because it is a $\omega$-tight and has uncountable dispersion character.

Consider a set $Z= \bigcup \{ U$, where  $U$ is an open subset of $X$ such that each  Lindel$\ddot{o}$f subspace of $U$ at most countable $\}$.
Then the set $Z$ is resolvable (as the union of resolvable subspaces \cite{Comfort}) or $Z=\emptyset$.

Let $Y=X \setminus [Z]$. By definition of $Y$ and Corollary \ref{col_2},
any open subset of $Y$ contains uncontable Lindel$\ddot{o}$f subspace $L$. If $L$ is a heriditaraly Lindel$\ddot{o}$f then it contains subspace $M$ of uncountable dispersion character. Then $M$ is resolvable (see \cite{JSS}).
Let  $H = \bigcup \{ M$, where $M$ is a heriditaraly Lindel$\ddot{o}$f  of uncountable dispersion character $\}$. The set $H$ is resolvable (as the union of resolvable subspaces \cite{Comfort}) or $H=\emptyset$.

Let $K=Y\setminus \overline{H}$. It remains to prove that $K$ is
resolvable. Let us note that if a subspace $L\subset K$ is not
heriditaraly Lindel$\ddot{o}$f, then it contains a closed set $F$
which is not $G_\delta$-set. Let $K\neq \emptyset$.

We claim that $K$ has the conditions of Corollary \ref{cor5}. Let
$V$ be an open set such that $\overline{V}\subset K$. Consider
$L\subset \overline{V}$  such that $L$ is Lindel$\ddot{o}$f, but
is not heriditaraly Lindel$\ddot{o}$f and $F\subset L$ such that
$F$ is a closed set of $L$, but is not $G_{\delta}$-set in $L$.

By induction on $\alpha$, we construct the required discrete
$D\subset L$.

Let $x_1\in L\setminus F$. By the regularity of $L$, there are an
open sets $U_1$ and $V_1$ such that $\overline{U_1}^{L}\cap
\overline{V_1}^{L}=\emptyset$, $x_1\in U_1$ and $F\subset V_1$.

Let we constructed $x_{\alpha}$, $V_{\alpha}$, $U_{\alpha}$ for
$\alpha<\beta$ such that $\overline{U_{\alpha}}^{L}\cap
\overline{V_{\alpha}}^{L}=\emptyset$, $x_{\alpha}\in U_{\alpha}$
and $F\subset V_{\alpha}$ and

1. $x_{\alpha}\in \bigcap\limits_{\gamma<\alpha}
\overline{V_{\gamma}^{L}}$;

2. $x_{\gamma}\notin U_{\alpha}$ for $\alpha\neq \gamma$;

3. $x_{\alpha}\notin F$;

4. $\bigcup\limits_{\gamma<\alpha} \{x_{\gamma}\}$ are closed
subsets of $L$ for $\alpha<\beta$.

If $\beta=\omega_1$ or $\bigcup\limits_{\gamma<\alpha}
\{x_{\gamma}\}$ is not closed set then inductive process is
completed.

If $\bigcup\limits_{\gamma<\alpha} \{x_{\gamma}\}$ is closed set
then there is $x_{\beta}\in \bigcap\limits_{\gamma<\beta}
\overline{V_{\gamma}}^{L}\setminus F$. There are an open sets
$V_{\beta}$ and $U_{\beta}$ such that $F\subset V_{\beta}$,
$x_{\beta}\in U_{\beta}$, $\overline{U_{\beta}}^{L}\cap
\overline{V_{\beta}}^{L}=\emptyset$, $\overline{V_{\beta}}^{L}\cap
(\bigcup\limits_{\gamma<\beta} \{x_{\gamma}\})=\emptyset$ and
$\overline{U_{\beta}}^{L}\cap (\bigcup\limits_{\gamma<\beta}
\{x_{\gamma}\})=\emptyset$.

By construction, if $|D|=\omega$ then $\overline{D}\setminus D\neq
\emptyset$.

If $|D|=\omega_1$, but $D$ contains a countable is not closed (in
$K$) subset $D_1$, then $D=D_1$.

If $|D|=\omega_1$ and each countable subset of $D$ is closed set
in $K$ then $D$ is a correct discrete set in Lindel$\ddot{o}$f
space $L$ and, hence, $D$ is a correct discrete set in
$\overline{V}$.

So $K$ is resolvable. Then $X$ is resolvable  as the union of
resolvable subspaces.
\end{proof}

Note that a regular $wL$-space is a $L$-space.

\begin{corollary} Let $\mathscr{P}\in \{\omega$N, $\omega$, $hL$, $\sigma$-$cmp$,
$ccc$, $L$,
 $wL \}$.  The regular $\mathscr{P}$-space with uncountable
dispersion character is
 resolvable.
\end{corollary}

\begin{lemma} Let $X$ be a regular resolvable space and $\Delta(X)>\omega$. Then there are disjoint dense in $X$ subsets $Y_1$, $Y_2$ of
$X$ such that $\Delta(Y_1)>\omega$.

\end{lemma}

\begin{proof} Let $X=X_1\cup X_2$ and $\overline{X_i}=X$ for
$i=1,2$. Suppose that $\Delta(X_1)=\Delta(X_2)=\omega$. Consider
an open set $U_1$ of $X$ such that $\Delta(U_1\cap X_1)=\omega$.
Then  $\Delta(U_1\cap X_2)>\omega$. Let $Z^1_1=U_1\cap X_2$,
$Z^2_1=U_1\cap X_1$. If $\Delta(X_1\setminus
\overline{U_1})>\omega$ then inductive process is completed.
Suppose that $\Delta(X_1\setminus \overline{U_1})=\omega$. Let
$U_2$ be an open set of $X$ such that $\Delta(X_1\setminus
\overline{U_2})=\omega$. $Z^1_2=U_2\cap X_2$, $Z^2_2=U_2\cap X_1$
and so on. By inductive process, we construct a disjoint family
$\{U_{\alpha}, \alpha < \beta \}$ of open subsets of $X$ such that
$X = [{\bigcup\limits_{\alpha<\beta}U_{\alpha}}]$ and disjoint
sets $Z^1_{\alpha}\subset U_{\alpha}$, $Z^2_{\alpha}\subset
U_{\alpha}$. Let $Y_1=\bigcup\limits_{\alpha < \beta}
Z^1_{\alpha}$ and $Y_2=\bigcup\limits_{\alpha<\beta}
Z^2_{\alpha}$. It is clear that $X=\overline{Y_i}$ for $i=1,2$ and
$\Delta(Y_1)>\omega$.
\end{proof}

\begin{theorem}  The regular $L$-tight space $X$ of uncountable dispersion character is $\omega$-resolvable.
\end{theorem}

\begin{proof}
% По теореме 9 пространство $X$ разложимо
  Let $X=Y_1\cup X_1,$ where $X_1,$ $Y_1$ are disjoint and dense in $X.$
  Let $\Delta(Y_1)>\omega$. % По лемме 2
  % В силу следствия 4,  $Y_1$ is $L$-space, следовательно, по теореме 9, $Y_1$ разложимо.
  Let $Y_1=Y_2\cup X_2,$ where $X_2,$ $Y_2$ are disjoint dense in $Y_1$ and $\Delta(Y_2)>\omega$. A space $Y_1$ is dense in $X$, consequently $X_2,$ $Y_2$ are dense in $X$ too. And so on.
  By inductive process, we construct a countable disjoint family
$\{X_n, \ n\in \omega\}$ of dense in $X$ sets.
\end{proof}

Note that a regular $wL$-tight is a $L$-tight.

\begin{corollary} Let $\mathscr{P}\in \{\omega$N, $\omega$, $hL$, $\sigma$-$cmp$, $ccc$,
$L$, $wL \}$.  The regular $\mathscr{P}$-tight space with
uncountable dispersion character is
 $\omega$-resolvable.
\end{corollary}


\begin{thebibliography}{1}
\bibitem{A_1} A. V. Arhangel'skii,
{\it A characterization of very $k$-spaces}, Czechoslovak
Mathematical Journal, 18 (1968), 392–-395.


\bibitem{A_S} A. V. Arhangel'skii and D. N. Stavrova,
{\it On a common generalization of k-spaces and spaces with
countable tightness}, Topology Appl., {\bf 51} (1993), 261--268.


\bibitem{Elk} A.G. El'kin, {\it Regular maximal spaces}, Math.
Notes, {\bf 27}:2, (1980), 150--151.

\bibitem{J_M} Istv\'{a}n Juh\'{a}sz, Jan van Mill,
{\it Variations on countable tightness},  arXiv:1702.03714v1 [math.GN], 13 Feb. (2017).

\bibitem{Hu43} E. Hewitt,
{\it A problem of set-theoretic topology},
Duke Math. J. {\bf 10} (1943), 309--333.

\bibitem{JSS} I.~Juhasz, L.~Soukup, Z.~Szentmiklossy,
{\it Regular spaces of small extent are $\omega$-resolvable}, Fundamenta Mathematicae, {\bf228}:1, (2015),  27--46.
DOI: 10.4064/fm228-1-3

\bibitem{Mal98} V.I.~Malykhin,
{\it Borel resolvability of compact spaces and their subspaces},
Mathematical Notes, {\bf64}:5, (1998), 607--615.
DOI: 10.1007/BF02316285

\bibitem{Vel} N.V.~Velichko,
{\it Theory of resolvable spaces}, Math. Notes, {\bf 19}:1,
(1976), 65--68.

\bibitem{Pyt} E.G.~Pytkeev,
{\it On maximally resolvable spaces},
Trudy Matematicheskogo Instituta Imeni V. A. Steklova,
{\bf 154}, (1983), 209--213. [in russian]
Zbl 0529.54005

\bibitem{Eng} R. Engelking,
{\it General Topology}, Heldermann Verlag, Berlin,1989.

\bibitem{Mic} E.A.~Michael,
{\it A quintuple quotient quest},
General Topology and Appl.,
{\bf 2}, (1972),91--138.
https://doi.org/10.1016/0016-660X(72)90040-2
PII: 0016-660X(72)90040-2

\bibitem{Bella_Mal} A.~Bella, VI~Malykhin,
{\it Tightness and resolvability},
Commentationes Mathematicae Universitatis Carolinae,
{\bf 39}:1, (1998), 177--184.
Persistent URL: http://dml.cz/dmlcz/118996

\bibitem{Sharma} P.L.~Sharma, S.~Sharma,
{\it Resolution properties in generalized k-spaces},
Topology and its Applications,
{\bf 29}, (1989), 61--66.

\bibitem{Comfort} W.W.~Comfort, L.~Feng,
{\it The union of resolvable spaces is resolvable},
Math. Japonica ,
{\bf 38}, (1993), 413--414.

\bibitem{Fi03} M.A.~Filatova,
{\it About resolvability of finally compact spaces},
Mathematicheskii i pricladnoi analis: sb. nauch. tr.--
Tumen : TSU, (2003), p. 204--212. [in russian]

\end{thebibliography}
\end{document}